\documentclass[12pt]{amsart}

\usepackage{ucs}
\usepackage{amssymb}
\usepackage{amsthm}
\usepackage{amsmath}
\usepackage{latexsym}
\usepackage[cp1251]{inputenc}
\usepackage[english]{babel}
\usepackage{graphicx}
	\usepackage{wrapfig}
\usepackage[justification=centering]{caption}
\usepackage{subcaption}
\usepackage{mathtools}
\usepackage[hidelinks]{hyperref}
\usepackage{tikz-cd}
\usepackage{amsbsy}
\usepackage{enumitem}
\usepackage{mathrsfs}
\usepackage{multicol}

\usepackage[left=2cm,right=2cm,top=2.5cm,bottom=2.5cm,bindingoffset=0cm]{geometry}

\title{Free Subshifts with Invariant Measures from the Lov\'asz Local Lemma}
\date{}
\author{Anton~Bernshteyn}
\thanks{Department of Mathematics, University of Illinois at Urbana--Champaign, IL, USA, \texttt{bernsht2@illinois.edu}. This research is supported by the Illinois Distinguished Fellowship.}

\newtheorem{theo}{Theorem}[section]
\newtheorem{prop}[theo]{Proposition}

\newtheorem{corl}[theo]{Corollary}

\newcommand*{\myproofname}{Proof}

\theoremstyle{definition}

\theoremstyle{remark}

\newcommand{\0}{\emptyset}
\newcommand{\set}[1]{\{#1\}}

\newcommand{\dom}{\operatorname{dom}}

\newcommand{\acts}{\curvearrowright}
\newcommand{\N}{\mathbb{N}}

\renewcommand{\epsilon}{\varepsilon}
\renewcommand{\phi}{\varphi}

\numberwithin{equation}{section}

\makeatletter
\newcommand{\neutralize}[1]{\expandafter\let\csname c@#1\endcsname\count@}
\makeatother

\begin{document}
	
	\maketitle
	
	\begin{abstract}
		Gao, Jackson, and Seward~\cite{GJS1} proved that every countably infinite group $\Gamma$ admits a nonempty free subshift $X \subseteq \set{0,1}^\Gamma$. Furthermore, a theorem of Seward and Tucker-Drob~\cite{T-DS} implies that every countably infinite group~$\Gamma$ admits a free subshift $X \subseteq \set{0,1}^\Gamma$ that supports an invariant probability measure. Aubrun, Barbieri, and Thomass\'e~\cite{ABT} used the Lov\'asz Local Lemma to give a short alternative proof of the Gao--Jackson--Seward theorem. Recently, Elek~\cite{E} followed another approach involving the Lov\'asz Local Lemma to obtain a different proof of the existence of free subshifts with invariant probability measures for finitely generated sofic groups. Using the measurable version of the Lov\'asz Local Lemma for shift actions established by the author in~\cite{B}, we give a short alternative proof of the existence of such subshifts for arbitrary groups. Moreover, we can find such subshifts in any nonempty invariant open set.
	\end{abstract}
	
	\section{Introduction}
	
	All group actions considered in this note are left actions. Let $\Gamma$ be a countable group. For an arbitrary set~$A$, the \emph{shift action} of $\Gamma$ on $A^\Gamma$ is defined as follows: For all $f \in A^\Gamma$ and $\gamma$, $\delta \in \Gamma$, let\footnote{There are two ways to define the shift action: either by multiplying on the right, or by multiplying on
		the left with the inverse. The latter one is more common; however, the former one is more convenient for our purposes.}
	$$
		(\gamma \cdot f)(\delta) \coloneqq f(\delta \gamma).
	$$
	Whenever $A$ is a topological space, the shift action of $\Gamma$ on $A^\Gamma$ is continuous (with respect to the product topology). A particular case of this is when $A$ is a countable set (assumed to be endowed with the discrete topology). If $A$ is finite, then a closed subset of $A^\Gamma$ invariant under the shift action is called a \emph{subshift}.
	
	It has been a matter of interest to establish which countable groups admit a nonempty subshift such that the induced action of $\Gamma$ on it is free (recall that an action $\Gamma \acts X$ is \emph{free} if for all $x \in X$ and for all $\gamma \in \Gamma$, $\gamma \cdot x = x$ implies that $\gamma$ is the identity element of $\Gamma$). Dranishnikov and Shroeder~\cite{DS} showed that any torsion-free hyperbolic group admits a free subshift. Shortly after, Glasner and Uspenskij~\cite{GU} proved that the same is true for groups that are either Abelian or residually finite. Finally, Gao, Jackson, and Seward~\cite{GJS1, GJS2} extended this result to all countable groups, thus completely solving the problem of existence of free subshifts.
	
	Seward and Tucker-Drob~\cite{T-DS} further developed the techniques of Gao--Jackson--Seward to establish the following very strong result: Whenever a countably infinite group $\Gamma$ is acting freely on a standard Borel space~$X$, there exists an equivariant Borel map $f \colon X \to \set{0,1}^\Gamma$ such that the action of $\Gamma$ on the closure of $f(X)$ is free. Thus, the closure of $f(X)$ is a nonempty free subshift; moreover, if the action of $\Gamma$ on $X$ admits an invariant probability measure, then so does the action of $\Gamma$ on $f(X)$. This implies that every countable group admits a free subshift supporting an invariant probability measure. 
	
	In~\cite{ABT}, Aubrun, Barbieri, and Thomass\'e gave a short alternative proof of the Gao--Jackson--Seward theorem on the existence of nonempty free subshifts. The key ingredient of their proof is the Lov\'asz Local Lemma (the~LLL for short), an immensely important tool in probabilistic combinatorics, that is often applied in order to establish the existence of combinatorial objects (such as colorings) satisfying certain constraints. In a very recent paper~\cite{E}, Elek used an approach based on nonrepetitive graph colorings and inspired by the ideas of~\cite{AGHR} to obtain a different proof of the existence of free subshifts with invariant probability measures in the case of finitely generated sofic groups. Elek's argument also relies heavily on the~LLL. The aim of this note is to use the \emph{measurable} version of the~LLL for shift actions, established by the author in~\cite{B}, to ``amplify'' the Auburn--Barbieri--Thomass\'e construction and get a very short alternative proof of the existence of free subshifts with invariant measures for arbitrary countable groups. Moreover, we show that such subshifts can be made disjoint from any given proper subshift:
	
	\begin{theo}\label{theo:subshifts}
		Let $\Gamma$ be a countably infinite group. If $Y \subsetneq \set{0,1}^\Gamma$ is a closed invariant proper subset, then there exists a closed invariant subset $X \subseteq \set{0,1}^\Gamma \setminus Y$  such that the action of $\Gamma$ on $X$ is free and admits an invariant probability measure.
	\end{theo}
	
	%The proof of Elek's result for sofic groups uses nonrepetitive graph colorings and is inspired by the ideas of~\cite{AGHR}. In particular, it involves the Lov\'asz Local Lemma (the~LLL for short), an immensely important tool in probabilistic combinatorics, that is often applied in order to establish the existence of combinatorial objects (such as colorings) satisfying certain constraints. The key ingredient of our proof of Theorem~\ref{theo:subshifts} is the \emph{measurable} version of the~LLL for shift actions proved by the author in~\cite{B}.
	
	\section{Proof of Theorem~\ref{theo:subshifts}}
	
	\subsection{The LLL and the measurable LLL}
	
	The general statement of the measurable LLL in \cite{B} is quite technical; however, we will only need a much simpler version of it. The definitions given below are special cases of the ones in \cite[Sections~1.2 and~5.1]{B}.
	
	Let $A$ be a countable set. Let $[A]^{<\infty}$ denote the set of all nonempty finite subsets of~$A$ and let $[A]^{<\infty}_{\set{0,1}}$ be the set of all functions $\phi \colon S \to \set{0,1}$, where $S \in [A]^{<\infty}$. An \emph{instance (of the~LLL)} over~$A$ is any subset ${B} \subseteq [A]^{<\infty}_{\set{0,1}}$. %We think of the elements of ${B}$ as ``bad'' functions that are to be avoided.
	For $S \in [A]^{<\infty}$, let $B_S \coloneqq \set{\phi \in {B} \,:\, \dom(\phi) = S}$. The \emph{domain} of ${B}$ is the set $\dom({B}) \coloneqq \set{S \in [A]^{<\infty} \,:\, {B}_S \neq \0}$. A~\emph{solution} for ${B}$ is a map $f \colon A \to \set{0,1}$ such that for all $S \in [A]^{<\infty}$, we have $f \vert S \not \in {B}$. The set of all solutions for an instance~${B}$ is denoted $\operatorname{Sol}({B})$. Note that $\operatorname{Sol}({B})$ is a closed subset of $\set{0,1}^A$.
	
	An instance ${B}$ is \emph{correct} if there is a function $p \colon \dom({B}) \to [0;1)$ such that for all $S \in \dom({B})$,
	\begin{equation}\label{eq:correct}
	\frac{|{B}_S|}{2^{|S|}} \leq \frac{p(S)}{1 - p(S)} \prod_{\substack{S' \in \dom({B}):\\ S' \cap S \neq \0}} (1 - p(S')).
	\end{equation}
	
	\begin{theo}[{A corollary of \textbf{the Lov\'asz Local Lemma}; Erd\H os--Lov\'asz}]\label{theo:LLL}
		Let $A$ be a countable set and let ${B}$ be a correct instance over $A$. Then $\operatorname{Sol}({B}) \neq \0$.
	\end{theo}
	
	See~\cite[Lemma~5.1.1]{AlonSpencer} for the full statement of the~LLL. Theorem~\ref{theo:LLL} is a particular case of the so-called \emph{variable version} of the~LLL (the name is due to Kolipaka and Szegedy~\cite{KSz}), which is the form in which the~LLL is usually applied. It is straightforward to deduce Theorem~\ref{theo:LLL} from the full~LLL when the set $A$ is finite (see~\cite[Corollary~1.7]{B}); the case of infinite $A$ then follows by compactness. A~more general version of Theorem~\ref{theo:LLL} for infinite $A$ was proved by Kun~\cite{Kun} using the Moser--Tardos algorithmic approach to the~LLL~\cite{MT}, see~\cite[Theorem~1.8]{B}. (The Moser--Tardos method also plays a crucial role in the proof of our measurable LLL.)
	
	From now on, let $\Gamma$ be a countably infinite group. We define the shift action of $\Gamma$ on the set $[\Gamma]^{<\infty}_{\set{0,1}}$ by declaring that for all $\phi \in [\Gamma]^{<\infty}_{\set{0,1}}$ and $\gamma$, $\delta \in \Gamma$, $\delta \in \dom(\gamma \cdot \phi)$ if and only if $\delta\gamma \in \dom(\phi)$, in which case
	$$(\gamma \cdot \phi)(\delta) \coloneqq \phi(\delta\gamma).$$
	An instance ${B}$ over $\Gamma$ is \emph{invariant} if it is closed under the shift action of $\Gamma$ on $[\Gamma]^{<\infty}_{\set{0,1}}$. Note that if ${B}$ is an invariant instance over $\Gamma$, then the set $\operatorname{Sol}({B})$ is invariant under the shift action of $\Gamma$ on $\set{0,1}^\Gamma$, and hence is a subshift.
	
	Let $\alpha \colon \Gamma \acts (X, \mu)$ be a measure-preserving action of~$\Gamma$ on a standard probability space $(X, \mu)$. A \emph{measurable solution} over $\alpha$ for an invariant instance ${B}$ is a Borel function $f \colon X' \to \set{0,1}$, defined on an invariant $\mu$-conull Borel subset $X'$ of $X$, such that for all $x \in X'$, the map 
	$$f_x \colon \Gamma \to \set{0,1} \colon \gamma \mapsto f(\gamma \cdot x)$$ belongs to $\operatorname{Sol}({B})$.
	
	\begin{prop}
		Let ${B}$ be an invariant instance over $\Gamma$ and let $\alpha \colon \Gamma \acts (X, \mu)$ be a measure-preserving action of $\Gamma$ on a standard probability space $(X, \mu)$. Suppose that ${B}$ admits a measurable solution over $\alpha$. Then the action of $\Gamma$ on $\operatorname{Sol}({B})$ admits an invariant probability measure.
	\end{prop}
	\begin{proof}
		Let $f$ be a measurable solution for ${B}$ over $\alpha$ and define a map $\pi \colon \dom(f) \to \operatorname{Sol}({B})$ by setting $\pi(x) \coloneqq f_x$ for all $x \in \dom(f)$. By definition, $\pi$ is equivariant, so the pushforward measure $\pi_\ast(\mu)$ is invariant, as desired.
	\end{proof}
	
	The following is a simplified special case of~\cite[Lemma~5.17]{B}, which, in turn, is a special case of \cite[Theorem~5.4]{B}:
	
	\begin{theo}[{\cite[Lemma~5.17]{B}}]\label{theo:meas_LLL}
		Let $\lambda$ denote the Lebesgue probability measure on $[0;1]$. If ${B}$ is a correct invariant instance over $\Gamma$, then ${B}$ admits a measurable solution over the shift action $\Gamma \acts ([0;1]^\Gamma, \lambda^\Gamma)$.
	\end{theo}
	
	\begin{corl}\label{corl:inv_meas}
		If ${B}$ is a correct invariant instance over $\Gamma$, then the action of $\Gamma$ on $\operatorname{Sol}({B})$ admits an invariant probability measure.
	\end{corl}
	
	Aubrun, Barbieri, and Thomass\'e~\cite{ABT} constructed a correct invariant instance $B$ over $\Gamma$ such that the action of $\Gamma$ on $\operatorname{Sol}({B})$ is free. In conjunction with Corollary~\ref{corl:inv_meas}, this implies that $\operatorname{Sol}(B)$ is a free subshift supporting an invariant probability measure. Below we present a modified version of the Auburn--Barbieri--Thomass\'e construction that also guarantees that $\operatorname{Sol}(B)$ is disjoint from a given subshift $Y$.
	
	\subsection{Constructing the instance}
	
	For $\phi \in [\Gamma]^{<\infty}_{\set{0,1}}$, let
	$$
		U_\phi \coloneqq \set{f \in \set{0,1}^\Gamma \,:\, f \vert \dom(\phi) = \phi}.
	$$
	Suppose that $Y \subsetneq \set{0,1}^\Gamma$ is a given closed invariant proper subset. Let $\psi \in [\Gamma]^{<\infty}_{\set{0,1}}$ be such that $Y \cap U_\psi = \0$. Since $Y$ is invariant, $Y \cap \gamma \cdot U_\psi = \0$ for all $\gamma \in \Gamma$. Set $k \coloneqq |\dom(\psi)|$.
	
	Let $N$ be a sufficiently large positive integer (to be determined later). Let $D_0$ be any subset of $\Gamma$ of size $N$ such that for all $\delta$, $\delta' \in D_0$, if $\delta \neq \delta'$, then $\dom(\psi) \delta \cap \dom(\psi) \delta' = \0$ (such $D_0$ exists since $\Gamma$ is infinite). Let $F_0 \coloneqq \dom(\psi) D_0$ and let $B_0'$ denote the set of all functions $\phi \colon F_0 \to \set{0,1}$ such that for all $\delta\in D_0$, $(\delta \cdot \phi)\vert \dom(\psi) \neq \psi$. A direct calculation shows that
	\begin{equation}\label{eq:zero_bound}
		\frac{|B_0'|}{2^{|F_0|}} = \left(1 - \frac{1}{2^k}\right)^N.
	\end{equation}
	Let $B_0 \coloneqq \Gamma \cdot B_0'$. By construction, $B_0$ is an invariant instance over $\Gamma$. Moreover, $\operatorname{Sol}(B_0) \cap Y = \0$. Indeed, suppose that $f \in \operatorname{Sol}(B_0)$. In particular, $f \in \operatorname{Sol}(B_0')$, i.e., there is some $\delta \in D_0$ such that $(\delta \cdot f) \vert \dom(\psi) =\psi$. This means that $f \in \delta^{-1} \cdot U_\psi$, and thus $f \not \in Y$.
	
	Let $M$ be another large integer (that will also be determined later) and let $(\gamma_n)_{n=1}^\infty$ be an enumeration of the non-identity elements of $\Gamma$ (with the numbering starting at~$1$). For each $n \geq 1$, fix a subset $D_n$ of $\Gamma$ of size $n+M$ such that $D_n \cap (D_n\gamma_n) = \0$ (this is again possible since $\Gamma$ is infinite). Set $F_n \coloneqq D_n \cup (D_n \gamma_n)$ and let ${B}'_n$ denote the set of all functions $\phi \colon F_n \to \set{0, 1}$ such that for all $\delta \in D_n$, $\phi(\delta) = \phi(\delta\gamma_n)$. Another straightforward calculation shows that for all $n \geq 1$,
	\begin{equation}\label{eq:one_bound}
		\frac{|B_n'|}{2^{|F_n|}} = \frac{1}{2^{n+M}}.
	\end{equation}
	Let ${B}_n \coloneqq \Gamma \cdot {B}_n'$ and ${B}_{\geq 1} \coloneqq \bigcup_{n=1}^\infty {B}_n$. Again, ${B}_{\geq 1}$ is an invariant instance over $\Gamma$, and, moreover, the action of $\Gamma$ on $\operatorname{Sol}(B_{\geq 1})$ is free. Indeed, suppose that $f \colon \Gamma \to \set{0,1}$ satisfies $\gamma_n \cdot f = f$ for some $n \geq 1$. In other words, for all $\delta \in \Gamma$, $f(\delta) = (\gamma_n \cdot f)(\delta) = f(\delta \gamma_n)$. This implies that $f \vert F_n \in {B}_n'$, i.e., $f \not \in \operatorname{Sol}({B}_{\geq 1})$.
	
	Finally, let $B \coloneqq B_0 \cup B_{\geq 1}$. Since $\operatorname{Sol}(B) = \operatorname{Sol}(B_0) \cap \operatorname{Sol}(B_{\geq 1})$, we have $\operatorname{Sol}(B) \cap Y = \0$ and the action of $\Gamma$ on $\operatorname{Sol}(B)$ is free. It only remains to show that $B$ is a correct instance. By definition, we have
	$$
		\dom(B) = \bigcup_{n=0}^\infty \dom(B_n) = \bigcup_{n=0}^\infty\set{F_n \gamma \,:\, \gamma \in \Gamma}.
	$$
	Notice that if $n \in \N$ and $S \in \dom(B_n)$, then for each $m \in \N$,
	\begin{equation}\label{eq:degree_bound}
		|\set{S' \in \dom(B_m)\,:\, S' \cap S \neq \0}| \leq |F_n|\cdot |F_m|.
	\end{equation}
	Indeed, let $S \eqqcolon F_n \gamma$, where $\gamma \in \Gamma$. If $S' = F_m\gamma'$ and $S' \cap S \neq \0$, then there exist some $\delta \in F_n$ and $\delta' \in F_m$ such that $\delta \gamma = \delta' \gamma'$. The choice of $\delta$ and $\delta'$ uniquely determines $\gamma'$, and thus~$S'$, and the number of such choices is exactly $|F_n| \cdot |F_m|$.
	
	To verify the correctness of ${B}$, we need to find a function $p \colon\dom({B}) \to [0;1)$ satisfying~\eqref{eq:correct} for all $S \in \dom({B})$. To that end, choose any positive real number $a < 1$ such that
	$$
		b \coloneqq a \cdot \left(1 - \frac{1}{k^2N^2}\right)^{2kN} > \frac{1}{2},
	$$
	and set
	$$
		p(S) \coloneqq \begin{cases}
			1/k^2 N^2  &\text{if }S \in \dom(B_0);\\
			a^{n + M} &\text{if } S \in \dom(B_n) \text{ for } n\geq 1.
		\end{cases}
	$$
	Due to~\eqref{eq:zero_bound}, \eqref{eq:one_bound}, and~\eqref{eq:degree_bound}, inequality~\eqref{eq:correct} is satisfied as long as we have
	\begin{equation}
	\label{eq:zero}
		\left(1 - \frac{1}{2^{k}}\right)^N \leq \frac{1}{k^2 N^2} \cdot \left(1 - \frac{1}{k^2 N^2}\right)^{k^2 N^2-1} \cdot \prod_{m = 1}^\infty (1 - a^{m+M})^{2(m+M)kN},
	\end{equation}
	and for all $n \geq 1$,
	\begin{equation}\label{eq:one}
		\frac{1}{2^{n + M}} \leq \frac{a^{n+M}}{1 - a^{n + M}} \cdot \left(1 - \frac{1}{k^2N^2}\right)^{2(n+M)kN}\cdot \prod_{m = 1}^\infty (1 - a^{m+ M})^{4(m+M)(n+M)}.
	\end{equation}
	Let us first deal with~\eqref{eq:zero}. Since $(1 - 1/d)^{d-1} \geq e^{-1}$ for all $d \geq 2$, we obtain
	$$
		\frac{1}{k^2 N^2} \cdot \left(1 - \frac{1}{k^2 N^2}\right)^{k^2 N^2-1} \geq \frac{1}{ek^2 N^2}.
	$$
	Choosing $N$ large enough, we can make sure that
	$$
		\left(1 - \frac{1}{2^{k}}\right)^N < \frac{1}{ek^2 N^2}.
	$$
	From now on, we will consider $N$ fixed. Since the infinite product $\prod_{m=1}^\infty(1 - a^m)^{2mkN}$ converges,
	$$
		\lim_{M \to \infty}\prod_{m = 1}^\infty (1 - a^{m+M})^{2(m+M)kN} = \lim_{M \to \infty} \prod_{m = M+ 1}^\infty (1 - a^m)^{2mkN} = 1,
	$$
	so~\eqref{eq:zero} holds for all sufficiently large $M$.
	
	Now we proceed to~\eqref{eq:one}. We have
	\begin{align*}
		\frac{a^{n+M}}{1 - a^{n + M}} \cdot \left(1 - \frac{1}{k^2N^2}\right)^{2(n+M)kN}\cdot \prod_{m = 1}^\infty (1 - a^{m+ M})^{4(m+M)(n+M)} 
		\geq \left(b \cdot \prod_{m = M+1}^\infty (1 - a^{m})^{4m} \right)^{n+ M}.
	\end{align*}
	Taking $M$ so large that $\prod_{m = M+1}^\infty (1 - a^m)^{4m} \geq 1/(2b)$, we obtain
	$$
		\left(b \cdot \prod_{m = M+1}^\infty (1 - a^{m})^{4m} \right)^{n+ M} \geq \frac{1}{2^{n+M}},
	$$
	thus establishing~\eqref{eq:one}.
	
	\subsection*{Acknowledgment}
	
	I am grateful to Robin Tucker-Drob for his helpful comments.
	
	%{\footnotesize
	
	%}
	
\end{document}